\documentclass{article}
\usepackage[utf8]{inputenc}
\usepackage[left=1in,top=1in,right=1in,bottom=1in]{geometry}
\usepackage[linktocpage=true]{hyperref}
\usepackage{setspace}
\usepackage{amssymb, amsmath, amsthm, graphicx,mathrsfs}
\usepackage{caption,cite}
\usepackage{mathtools,color}

\usepackage{cleveref}
\usepackage{tikz}
\usetikzlibrary{shapes.geometric, backgrounds, positioning, arrows.meta,calc}

\newtheorem{thm}{Theorem}[section]

\newtheorem{defn}[thm]{Definition}

\newtheorem{lem}[thm]{Lemma}
\newtheorem{conj}[thm]{Conjecture}

\newtheorem{example}[thm]{Example}

\newcommand{\F}{\mathcal{F}}
\newcommand{\cH}{\mathcal{H}}
\renewcommand{\P}{\mathcal{P}}
\newcommand{\B}{\mathcal{B}}
\newcommand{\eps}{\varepsilon}
\renewcommand{\l}{\left}
\renewcommand{\r}{\right}

\title{Matching and intersection problems for\\ non-trivial $r$-partite $r$-uniform hypergraphs}
\author{Peter Frankl\footnote{Alfréd R\'{e}nyi Institute of Mathematics, Budapest, Hungary. Email: \texttt{frankl.peter@renyi.hu}.}\quad\quad
Jiaxi Nie\footnote{Corresponding Author. School of Mathematics, Georgia Institute of Technology. Email: \texttt{jnie47@gatech.edu}.}}

\begin{document}

\maketitle

\begin{abstract}

A central theme in extremal combinatorics is the study of the maximum number of edges in an $r$-uniform hypergraph ($r$-graph) with matching number at most $s$ (the Erd\H{o}s Matching Conjecture) or with pairwise intersection at least $t$ (the $t$-intersection problem).
The maximum sizes for these problems are typically achieved by trivial constructions: for the matching problem, the extremal construction consists of all edges intersecting a fixed set of $s$ vertices, while for the intersection problem, it consists of all edges containing a fixed set of $t$ vertices.

In this paper, we investigate the \emph{non-trivial} $r$-partite $r$-graphs where each part is of size $n$. We determine the exact bounds for both the matching problem and the intersection problem when $n$ is sufficiently large. 
Furthermore, for the intersection problem, we resolve the cases $t=1$ and $t=r-2$ for all $n \ge 2$. Our results partially confirm a conjecture of Lu and Ma (``Matching Stability for 3‐Partite 3‐Uniform Hypergraphs.'' Journal of Graph Theory (2026)). 
\end{abstract}

\medskip

% \noindent\textbf{Keywords: $r$-partite $r$-uniform hypergraph, Erd\H{o}s Matching Conjecture, Erd\H{o}s-Ko-Rado Theorem, Hilton-Milner Theorem}

\section{Introduction}
Let integers $r \ge 2$, $n_1 \ge n_2 \ge \dots \ge n_r \ge 2$, and let $X_\ell = \{1, 2, \dots, n_\ell\}, 1 \le \ell \le r$. 
An \emph{$r$-partite $r$-graph} with partite sets $X_1, \dots, X_r$ is a family $\mathcal{F} \subseteq X_1 \times \dots \times X_r = \{(x_1, \dots, x_r) : x_\ell \in X_\ell, \, 1 \le \ell \le r\}$. 
We refer to the full product $X_1 \times \dots \times X_r$ as the \emph{complete $r$-partite $r$-graph}. 
To visualize this $r$-graph the reader might want to consider $X_1,\dots,X_r$ as pairwise disjoint sets, i.e. copies of $[n_\ell]=\{1,\dots,n_\ell\}$.
In this setting, an edge $A \in \mathcal{F}$ corresponds to a set $A \subseteq X_1 \cup \dots \cup X_r$ such that $|A \cap X_\ell| = 1$ for all $1 \le \ell \le r$.

For a vector $A \in X_1 \times \dots \times X_r$ and an index $1 \le \ell \le r$, let $A[\ell]$ denote the $\ell$-th coordinate of $A$. 
Two edges $A$ and $B$ are \emph{disjoint} if $A[\ell] \neq B[\ell]$ for all $1 \le \ell \le r$. 
A collection of pairwise disjoint edges is called a \emph{matching}. 
The \emph{matching number} of a family $\mathcal{F}$, denoted by $\nu(\mathcal{F})$, is the maximum size of a matching contained in $\mathcal{F}$. 
A subset $T \subseteq X_1 \cup \dots \cup X_r$ is called a \emph{transversal} of $\mathcal{F}$ if every edge in $\mathcal{F}$ contains at least one vertex from $T$. 
For example, any partite set $X_\ell$ is a transversal. 
The minimum size of a transversal is denoted by $\tau(\mathcal{F})$.

The classical K\"onig's Theorem~\cite{Konig1931} asserts that for $r=2$, $\nu(\mathcal{F}) = \tau(\mathcal{F})$. 
For $r\ge 3$ the situation is much more complicated. The corresponding extremal problem can be stated as follows. 

\begin{defn}
For integers $r \ge 2$ and $n_1 \ge \dots \ge n_r\ge s\ge 1$, let $m(s, n_1, \dots, n_r)$ denote the maximum size of a family $\mathcal{F} \subseteq X_1 \times \dots \times X_r$ satisfying $\nu(\mathcal{F}) \le s$. 
In the symmetric case where $n_1 = \dots = n_r = n$, we write $m(s, n; r) = m(s, \underbrace{n, \dots, n}_{r})$.
\end{defn}

Meyer\cite{meyer1974quelques} proved $m(1,n;r)=n^{r-1}$ and it was generalized by Deza and Frankl~\cite{deza1983erdos} to $m(1,n_1,\dots,n_r)\le \prod_{\ell=1}^{r-1}n_\ell$. In~\cite{frankl2012disjoint}, Frankl proved for all $s\le n_r$
\begin{equation}\label{eq:s-matchingEKR}
m(s,n_1,\dots,n_r)\le s\prod_{\ell=1}^{r-1}n_\ell.
\end{equation}

It should be clear that these bounds are best possible. Namely, fixing $T\subset X_r$ with $|T|=s$ and setting $\F_T=\{(x_1,\dots,x_r)\in X_1\times\dots\times X_r: x_r\in T\}$ provides equality in (\ref{eq:s-matchingEKR}). As $\tau(\F_T)=s$, it is natural to consider the following notion.

\begin{defn}
For integers $r\ge 3$ and $n_1\ge n_2\ge\dots\ge n_r>s\ge 1$, let $m_0(s, n_1, \dots, n_r)$ denote the maximum size of a family $\mathcal{F} \subseteq X_1 \times \dots \times X_r$ satisfying $\nu(\mathcal{F}) \le s<\tau(\F)$. In the symmetric case where $n_1 = \dots = n_r = n$, we write $m_0(s,n;r)=m_0(s,\underbrace{n,\dots,n}_{r})$.
\end{defn}

The corresponding $r$-partite $r$-graphs are often called \textit{non-trivial}. It should be mentioned that the analogous problem for $\binom{[n]}{r}$, that is, for $r$-subsets of an $n$-set has been widely investigated. In particular the classical Erd\H{o}s-Ko-Rado Theorem~\cite{ErdosKoRado1961} states that $|\F|\le \binom{n-1}{r-1}$ for $\F\subset\binom{[n]}{r}$ if $n\ge 2r$ and $\nu(\F)=1$. Further imposing the restriction $\tau(\F)>1$, Hilton and Milner~\cite{HiltonMilner1967} established the best possible bound
\begin{equation*}
|\F|\le \binom{n-1}{r-1}-\binom{n-r-1}{r-1}+1.
\end{equation*}

The full solution of the corresponding problem for $s\ge 2$ would include the resolution of the Erd\H{o}s Matching Conjecture~\cite{Erdos1965Matching} and is not in sight. 

However the case of $r$-partite $r$-graphs might be easier. In a recent paper Lu and Ma~\cite{lu2024matching} proposed the following conjecture.

\begin{conj}[Conjecture 1.4 in \cite{lu2024matching}]\label{conj:matching}
Let $n$ be sufficiently large, $r\ge 4$ and $s<n$. Then
\begin{equation}\label{eq:conjecture}
m_0(s,n,\dots,n)=sn^{r-1}-(n-1)^{r-1}+n-s.
\end{equation}
\end{conj}

For $F\in\F$ define its \textit{projection} $\P(F)=\{\ell:F[\ell]=1\}$. Further, define $\mathcal{P}(\mathcal{F})=\{\mathcal{P}(F):F\in\F\}$.

\begin{defn}\label{defn:extension}
Let integers $r\ge 3$ and $n_1\ge \dots\ge n_r\ge 2$. Given a hypergraph $\mathcal{H} \subseteq 2^{[r]}$, we define the family $\mathcal{H}(n_1, \dots, n_r) \subseteq [n_1] \times \dots \times [n_r]$ such that a vector $F$ belongs to $\mathcal{H}(n_1, \dots, n_r)$ if and only if there exists an edge $E \in \mathcal{H}$ satisfying $E \subseteq \mathcal{P}(F)$. 
In the symmetric case where $n_1 = \dots = n_r = n$, we write $\mathcal{H}(n) := \mathcal{H}(\underbrace{n,\dots,n}_{r})$.
\end{defn}
In other words, $\cH(n_1,n_2,\dots,n_r)$ is the maximum $r$-partite $r$-graph in $[n_1]\times\dots\times[n_r]$ such that its projection contains $\cH$. 

For $r\ge 3$, let
$$
W_r=\{\{\ell,r\}:1\le \ell\le r-1\}\cup\{\{1,\dots,r-1\}\}.
$$

The construction corresponding to Conjecture \ref{conj:matching} is building upon $W_r(n)$, the natural analogue of the Hilton-Milner $r$-graph.

\begin{example}
Let integer $r\ge 3$ and let $X_1,\dots,X_r$ be disjoint copies of $[n]$. Let $\mathcal{E}(r,s,n)$ consist of the following sequences $(a_1,\dots,a_r)$: all sequences with $a_r\le s$ except for $(a_1,\dots,a_r)$ satisfying $a_1,\dots,a_{r-1}\ge 2$ and $a_r=1$ along with all $(1,1,\dots,1,b_r)$ with $s+1\le b_r\le n$. Alternatively,
$$
\mathcal{E}(r,s,n)=W_r(n)\cup(\bigcup_{i=2}^s (X_1\times\dots\times X_{r-1}\times \{i\})).
$$
\end{example}

Lu and Ma~\cite{lu2024matching} confirmed the statement of~Conjecture \ref{conj:matching} for $r=3$; more precisely, they proved (\ref{eq:conjecture}) for $r=3$, $s\ge 1$ and $n> \max\{162,s\}$. In this paper we confirm~Conjecture \ref{conj:matching} for $r\ge 3$ and $s=1$ in a stronger sense.

\begin{thm}[Proof is in Section~\ref{sec:shifting}]\label{thm:s=1}
For integers $r\ge 3$ and $n\ge 2$,
\begin{equation}
m_0(1,n,\dots,n)=n^{r-1}-(n-1)^{r-1}+n-1.
\end{equation}
\end{thm}

We also confirm~Conjecture \ref{conj:matching} for any fixed $r\ge3$ and $s\ge 2$.
\begin{thm}[Proof is in Section~\ref{sec:sunflower}]\label{thm:matching_n_large}
For integers $r\ge 3$ and $s\ge 2$, there exists $n(r,s)$ such that if $n>n(r,s)$, then
\begin{equation}
m_0(s,n,\dots,n)=sn^{r-1}-(n-1)^{r-1}+n-s.
\end{equation}
\end{thm}

Moreover, in the case when $r=3$ and $s=1$, we solve the problem when the parts are not necessarily of the same size.

\begin{thm}[Proof is in Section~\ref{sec:r-2}]\label{thm:3-partite intersecting}
For all integers $n_1\ge n_2\ge n_3\ge 2$, 
$$
m_0(1,n_1,n_2,n_3)=n_1+n_2+n_3-2.
$$
\end{thm}

Note that $m_0(1,n_1,n_2,n_3,n_4)$ has no ``universal'' answer, even when $n_1\le n_4+1$ and $n_4$ is sufficiently large. See Theorem~\ref{thm:matching_n_large'}; the value of $m_0(1,n_1,n_2,n_3,n_4)$ takes one of the two forms depending on whether $(n_1-1)(n_4-1)\le (n_2-1)(n_3-1)$ or not. For $r\ge 5$, the situation is even more complex and it would be interesting to investigate.

A family $\mathcal{F} \subseteq \binom{[n]}{r}$ is called \emph{$t$-intersecting} if any two sets in $\mathcal{F}$ intersect in at least $t$ elements. 
In their seminal work, Erd\H{o}s, Ko, and Rado~\cite{ErdosKoRado1961} proved that $|\mathcal{F}| \le \binom{n-t}{r-t}$ whenever $n$ is sufficiently large. 
This maximum is achieved by the family consisting of all $r$-subsets containing a fixed set of size $t$. 
Frankl~\cite{frankl1978erdos} conjectured that the maximum is always achieved by the family of all $r$-subsets intersecting some fixed set of size $t+2i$ in at least $t+i$ elements, for some integer $0 \le i \le r-t$. 
This conjecture was completely settled by Ahlswede and Khachatrian~\cite{ahlswede1997complete}.

Let $X_1, \dots, X_r$ be sets, and consider sequences $A=(a_1, \dots, a_r)$ and $B=(b_1, \dots, b_r)$ in $X_1 \times \dots \times X_r$. 
We define the intersection of two such sequences as the set of indices where they agree: $A \cap B := \{ \ell \in [r] : a_\ell = b_\ell \}$. 
A family $\mathcal{F} \subseteq X_1 \times \dots \times X_r$ is called \emph{$t$-intersecting} if $|A \cap B| \ge t$ for all $A, B \in \mathcal{F}$. The $t$-intersection problem for $r$-partite $r$-graphs can be stated as follows. 

\begin{defn}
For integers $r>t \ge 1$ and $n_1 \ge \dots \ge n_r\ge 2$, let $\iota(t, n_1, \dots, n_r)$ denote the maximum size of a $t$-intersecting family $\mathcal{F} \subseteq X_1 \times \dots \times X_r$. 
In the symmetric case where $n_1 = \dots = n_r = n$, we write $\iota(t, n; r) = \iota(t, \underbrace{n, \dots, n}_{r})$.
\end{defn}

For a family $\mathcal{F} \subseteq X_1 \times \dots \times X_r$, we define the set of fixed coordinates as
\[
\bigcap\mathcal{F} := \{ \ell \in [r] : \text{all sequences in } \mathcal{F} \text{ have the same entry in the } \ell\text{-th coordinate} \}.
\]
A $t$-intersecting family $\mathcal{F}$ is called \emph{non-trivial} if $|\bigcap\mathcal{F}| < t$. 
Note that non-triviality has different meanings in the contexts of matching problems and intersection problems, and the reader should distinguish between them accordingly.

Frankl and Füredi~\cite{frankl1980erdos} proved that $\iota(t, n; r) = n^{r-t}$ for $t \ge 15$ and sufficiently large $n$. 
They conjectured that the same equality holds for all $t \ge 1$ and $n \ge t+1$. 
This conjecture was completely resolved by Frankl and Tokushige~\cite{frankl1999erdos}. Note that the maximum is indeed achieved by a trivial family. It is natural to ask for the maximum when the trivial family is not allowed.

\begin{defn}
For integer $r>t\ge 1$ and $n_1 \ge \dots \ge n_r\ge2$, let $\iota_0(t, n_1, \dots, n_r)$ denote the maximum size of a $t$-intersecting family $\mathcal{F} \subseteq X_1 \times \dots \times X_r$ with $|\bigcap\F|<t$. 
In the symmetric case where $n_1 = \dots = n_r = n$, we write $\iota_0(t, n; r) = \iota_0(t, \underbrace{n, \dots, n}_r)$.
\end{defn}
Note that $m_0(1,n_1,\dots,n_r)=\iota_0(1,n_1,\dots,n_r)$. We prove the following generalization of Theorem \ref{thm:3-partite intersecting}.

\begin{thm}[Proof is in Section~\ref{sec:r-2}]\label{thm:(r-2)-intersecting}
For integers $r\ge 3$ and $n_1\ge \dots\ge n_r\ge 2$, 
$$
\iota_0(r-2,n_1,\dots,n_r)=n_1+\dots+n_r-r+1.
$$
\end{thm}

In the case when all parts have the same size $n$, and when $n$ is sufficiently large, we determine $\iota_0(t,n;r)$, showing that its value exhibits a phase transition at the threshold $t=\frac{r}{2}-1$.
\begin{thm}[Proof is in Section~\ref{sec:intersection}]\label{thm:t-intersecting}
For integers $r\ge 3$ and $r-2\ge t\ge 1$, there exists an integer $n(r,t)$ such that if $n\ge n(r,t)$, then
$$
\iota_0(t,n;r)=\max\{n^{r-t}-(n-1)^{r-t}+t(n-1),(t+2)n^{r-t-1}-(t+1)n^{r-t-2}\}.
$$
Note that, given $n$ is sufficiently large, $\iota_0(t,n;r)$ takes the first value in the max above when $t\le \frac{r}{2}-1$, and takes the second otherwise.
\end{thm}

Define 
$$
W_{r,t}:=\{[t]\cup\{\ell\}:t+1\le \ell\le r \}\cup\{[r]\setminus\{\ell\}:1\le \ell\le t\},
$$
and $K_{r,t}:=\binom{[t+2]}{t+1}\in 2^{[r]}$. Note that $K_{r,t}$ should be viewed as a complete $(t+1)$-graph on $t+2$ vertices union with $r-t-2$ isolated vertices. The first value in the max in Theorem~\ref{thm:t-intersecting} is achieved by $W_{r,t}(n)$ while the second is achieved by $K_{r,t}(n)$. 

In an earlier version of this paper, we conjectured that the formula in Theorem~\ref{thm:t-intersecting} holds for every $n\ge 2$. Shortly thereafter, Václav Rozhoň, Robert Samal, and Adrián Zámečník informed us that they had found counterexamples using Bolzano~\cite{bolzano2026casestudies}, an AI-assisted research tool. Subsequently, Hou and Hu~\cite{hou2026non} determined $\iota_0(t,n;r)$ for all $n\ge 2$.

\section{Shifting coordinate-wise}\label{sec:shifting}
The shifting technique is a powerful tool in extremal set theory,  usually used to simplify a family of sets while preserving key properties. It is first used by Erd\H{o}s, Ko, and Rado~\cite{ErdosKoRado1961} to prove the seminal Erd\H{o}s-Ko-Rado
Theorem. See~\cite{frankl1987shifting} for a survey on this technique. In this section we use the shifting technique to prove Theorem \ref{thm:s=1}.

Let us recall that $\F\subset X_1\times\dots\times X_r$ is called $t$-intersecting if for all $(a_1,\dots,a_r), (b_1,\dots,b_r)\in\F$, $a_\ell=b_\ell$ holds for at least $t$ distinct values of $1\le \ell\le r$. To avoid trivialities we tacitly assume $r>t\ge 1$. Let $1\le \ell\le r$ be fixed and choose $1<j\le n_\ell$. Let us define the $(1\leftarrow j)$-shift $S^{(\ell)}_j$ by $S^{(\ell)}_j(\F)=\{S_j^{(\ell)}(F):F\in\F\}$ where
$$
S^{(\ell)}_j((a_1,\dots,a_\ell,\dots,a_r))=
\left\{
\begin{aligned}
&F'=(a_1,\dots,1,\dots,a_r)&\text{~if~}a_{\ell}=j\text{~and~}F'\not\in\F,\\
&(a_1,\dots,a_r)&\text{~otherwise}.
\end{aligned}
\right.
$$
In human language, if $(a_1,\dots,a_r)\in\F$ and $a_{\ell}=j$ then we try and replace the $\ell^{\text{th}}$ coordinate by $1$ but refrain from doing so if the altered sequence is already an edge. This ensures $|S^{(\ell)}_j(\F)|=|\F|$. In~\cite{frankl1980erdos} it was shown that the $(1\leftarrow j)$-shift maintains the $t$-intersecting property. We say that $\F$ is \textit{$\ell$-shifted} if $S^{(\ell)}_j(\F)=\F$ for all $j$. Further, call $\F$ \textit{coordinate-wise shifted} if it is $\ell$-shifted for every $1\le \ell\le r$.

\begin{lem}\label{lem:projection_maintain_intersection}
Let integers $r\ge 3$ and $r-2\ge t\ge 1$. If $\F\subset X_1\times \dots \times X_r$ is coordinate-wise shifted and non-trivial $t$-intersecting, then the family of projections $\mathcal{P}(\mathcal{F})$ is non-trivial $t$-intersecting as well. 
\end{lem}

\begin{proof}
We first prove the $t$-intersecting property. Suppose for contradiction that there exist $A,B\in\P(\F)$ such that $|A\cap B|<t$. By definition, there exist $F_A,F_B\in\F$ with $\P(F_A)=A$ and $\P(F_B)=B$. Let $F'_A$ be the sequence obtained by replacing the $\ell^{\text{th}}$ coordinate of $F_A$ by 1 for every $\ell \in [r]\setminus(A\cup B)$. Because $\F$ is coordinate-wise shifted we have $F'_A\in\F$. It is not hard to check that $F'_A\cap F_B=A\cap B$. Thus $|F'_A\cap F_B|=|A\cap B|<t$, which contradicts $\F$ being $t$-intersecting. Thus $\P(\F)$ is $t$-intersecting.

Next, suppose for contradiction that $|\bigcap\P(\F)|=t$. Note that if $\ell\in \bigcap\P(\F)$, then for every $F\in\F$, we have $\ell\in\P(F)$  i.e. $F[\ell]=1$, which implies $\ell\in\bigcap \F$. Thus, $|\bigcap\F|\ge |\bigcap\P(\F)|=t$, contradicting to $\F$ being non-trivial. Therefore, $\P(\F)$ is non-trivial $t$-intersecting.
\end{proof}

There is one major problem with shifting, namely it might decrease the covering number $\tau$, in particular it might turn a non-trivial family into a trivial one. We say a non-trivial $t$-intersecting family $\F\subset X_1\times \dots \times X_r$ is \textit{$\ell$-shift-resistant} if there exists $x\in X_\ell$ such that $S_x^{(\ell)}(\F)$ is trivial, i.e. $|\bigcap S_x^{(\ell)}(\F)|=t$. For $\F\subset X_1\times\dots\times X_r$ and $x\in X_\ell$, define 
$$N_{\F}(x,\ell):=\{A\in\F: A[\ell]=x \}.$$
\begin{lem}\label{lem:shift_resistant}
Let integers $r\ge 3$ and $r-2\ge t\ge 1$. Let $\F\subset X_1\times\dots\times X_r$ be non-trivial $t$-intersecting and $\ell$-shift-resistant for some $\ell\in[r]$. Then $|\bigcap F|=t-1$ and there exists $x\in X_\ell$ such that $N_{\F}(x,\ell)\not=\emptyset$, $N_{\F}(1,\ell)\not=\emptyset$, $S_x^{(\ell)}(N_{\F}(x,\ell))\cap N_{\F}(x,\ell)=\emptyset$, and $N_{\F}(y,\ell)=\emptyset$ for every $y\in X_\ell\setminus\{x,1\}$.
\end{lem}

\begin{proof}
 Since only the $\ell$-th coordinate is affected by $S_x^{(\ell)}$, the intersections $\bigcap\F$ and $\bigcap S_x^{(\ell)}(\F)$ can differ only by one element $\ell$. Thus $\F$ being $\ell$-shift-resistant, i.e. $|\bigcap\F|<|\bigcap S_x^{(\ell)}(\F)|=t$ for some $x\in X_\ell$, implies that $\ell\not\in\bigcap\F$, $\bigcap S_x^{(\ell)}(\F)=(\bigcap\F)\cup\{\ell\}$, and hence $|\bigcap\F|=t-1$. It is clear from the fact $\ell\in \bigcap S_x^{(\ell)}(\F)\setminus \bigcap\F$ that all sequences in $S_x^{(\ell)}(\F)$ has entry $1$ at their $\ell$-th coordinate, and all sequences in $\F$ has entry $1$ or $x$ at their $\ell$-th coordinate, which implies that $N_{\F}(y,\ell)=\emptyset$ for every $y\in X_\ell\setminus\{x,1\}$. 
 
 Moreover, neither $N_{\F}(1,\ell)$ nor $N_{\F}(x,\ell)$ can be empty; otherwise all members of $\F$ would share the same value at the $\ell$-th coordinate, contradicting $\ell\not\in \bigcap\F$. 
 
 Finally, suppose $S_x^{(\ell)}(N_{\F}(x,\ell))\cap N_{\F}(x,\ell)\not=\emptyset$, then some $A\in N_{\F}(x,\ell)$ remains unchanged after the $(x\rightarrow1)$-shift, contradicting $\ell\in \bigcap S_x^{(\ell)}(\F)$. 
 
 This completes the proof.
\end{proof}

\begin{lem}\label{lem:not_shifted}
Let integers $r\ge 3$ and $r-2\ge t\ge 1$. Let $|X_\ell|=n\ge 2$ for every $1\le \ell \le r$. Let $\F\subset X_1\times\dots\times X_r$ be non-trivial $t$-intersecting such that $\F$ is either $\ell$-shifted or $\ell$-shift-resistant for every $1\le \ell\le r$, and $\F$ is not coordinate-wise shifted. Then
$$
|\F|\le (t+2)n^{r-t-1}-(t+1)n^{r-t-2}.
$$
\end{lem}

\begin{proof}
Let $b$ be the number of elements $\ell\in[r]$ such that $\F$ is $\ell$-shift-resistant. We have $b\ge 1$ because $\F$ is not coordinate-wise shifted. 

Suppose for the sake of a  contradiction that \(b=1\). Without loss of generality, suppose \(\mathcal{F}\) is \(1\)-shift-resistant. By Lemma~\ref{lem:shift_resistant} there exists
\(x\in X_1\) with \(x\neq 1\) such that \(N_{\mathcal{F}}(x,1)\neq\emptyset\). Since \(\mathcal{F}\) is
\(\ell\)-shifted for every \(\ell\neq 1\), both \((1,1,\dots,1)\) and \((x,1,\dots,1)\) belong to
\(\mathcal{F}\). Thus $S^{(1)}_x\bigl((x,1,\dots,1)\bigr)=(x,1,\dots,1)$, and hence \(S^{(1)}_x(N_\mathcal{F}(x,1))\cap N_\mathcal{F}(x,1)\neq\emptyset\), contradicting Lemma~\ref{lem:shift_resistant}.
Therefore \(b\ge 2\).

Suppose \(\mathcal{F}\) is \(\ell\)-shift-resistant for all \(1\le \ell\le b\). By
Lemma~\ref{lem:shift_resistant}, for each \(1\le \ell\le b\), there exists
\(x_\ell\in X_\ell\) with \(x_\ell\neq 1\) such that \(F[\ell]\in\{1,x_\ell\}\) for every
\(F\in\mathcal{F}\). We count the number of ways to specify an element \(F\in\mathcal{F}\) by choosing its
coordinates, leaving \(F[1]\) to be determined last. For each \(2\le \ell\le b\), there are two
possibilities for \(F[\ell]\). Since \(\bigcap\mathcal{F} = t-1\) by Lemma~\ref{lem:shift_resistant},
among the coordinates \(b+1\le \ell\le r\), exactly \(t-1\) are fixed, and each of the remaining
\(r-b-t+1\) coordinates can be chosen in at most \(n\) ways. Once all coordinates \(F[\ell]\) with \(2\le \ell\le r\) are specified, the value of \(F[1]\) is
uniquely determined. Indeed, Lemma~\ref{lem:shift_resistant} implies that
\(S^{(1)}_{x_1}(N_\mathcal{F}(x_1,1))\cap N_\mathcal{F}(x_1,1)=\emptyset\), so no two members of \(\mathcal{F}\) may
agree in every coordinate except the first. Therefore,
\[
|\mathcal{F}|\le 2^{\,b-1} n^{\,r-b-t+1}.
\]

Since \(n\ge 2\) and \(t\ge 1\), we have \(t(n-1)\ge 1\). Therefore,
\[
|\mathcal{F}|\le 2n^{\,r-t-1}
   \le 2n^{\,r-t-1} + n^{\,r-t-2}(tn - t - 1)
   = (t+2)n^{\,r-t-1} - (t+1)n^{\,r-t-2}.
\]
\end{proof}

With Lemma~\ref{lem:not_shifted}, the $t$-intersecting problem is now reduced to the case when $\F$ is coordinate-wise shifted. Let us first consider the $t=1$ case.
\begin{lem}\label{lem:I1}
For integers $n\ge 2$ and $r\ge 3$,
$$n^{r-1}-(n-1)^{r-1}+n-1\ge3n^{r-2}-2n^{r-3}.$$
\end{lem}

\begin{proof}
We prove by induction on $r\ge 3$. Note that
$$
n^{r-1}-(n-1)^{r-1}+n-1-3n^{r-2}+2n^{r-3}=(n-1)((n-2)n^{r-3}-(n-1)^{r-2}+1).
$$
Since $n\ge 2$, it suffices to show that 
$$
(n-2)n^{r-3}\ge (n-1)^{r-2}-1.
$$
It is clear that the inequality holds for $r=3$. For $r\ge 4$, by inductive assumption
\begin{align*}
(n-2)n^{r-3}&=n(n-2)n^{r-4}\\
&\ge n((n-1)^{r-3}-1)\\
&=(n-1)^{r-2}+(n-1)^{r-3}-n\\
&\ge (n-1)^{r-2}-1.
\end{align*}
This proves the desired inequality
\end{proof}

% Note that for any $R\subset[r]$,
% \begin{equation}
% |\{F\in X_1\times\dots\times X_r:\P(F)=R\}|=(n-1)^{r-|R|}
% \end{equation}

\begin{lem}\label{lem:shifted}
Let integer $r\ge 3$. Let $|X_\ell|=n\ge 2$ for every $1\le \ell \le r$. Let $\F\subset X_1\times\dots\times X_r$ be coordinate-wise shifted and non-trivial intersecting and of maximum size. Then $|\F|\le n^{r-1}-(n-1)^{r-1}+n-1.$
\end{lem}

\begin{proof}
By Lemma \ref{lem:projection_maintain_intersection}, $\P(\F)\subset 2^{[r]}$ is non-trivial intersecting. Moreover, since $\F$ is of maximum size, we know $\F$ contains all edges whose projection is in $\P(\F)$. Thus

\begin{equation}
|\F|=\sum_{P\in\P(\F)} (n-1)^{r-|P|}.
\end{equation}

Let $\rho_{i}$ denote the number of $i$-sets in $\P(\F)$. By the intersecting property and the Erd\H{o}s-Ko-Rado Theorem, the following two inequalities are immediate.

\begin{equation}\label{eq:i&r-i}
\rho_{i}+\rho_{r-i}\le \binom{r}{i},
\end{equation}

\begin{equation}\label{eq:i<r/2}
\rho_{i}\le \binom{r-1}{i-1}~~\text{for}~~2\le i\le \frac{r}{2}.
\end{equation}

The nontriviality implies

\begin{equation}\label{eq:0&1}
\rho_{1}=\rho_{0}=0.
\end{equation}

From (\ref{eq:i&r-i}) and (\ref{eq:i<r/2}) we infer that for $i\le r/2$,
$$
\begin{aligned}
\rho_{i}(n-1)^{r-i}+\rho_{r-i}(n-1)^i
&\le \rho_{i}((n-1)^{r-i}-(n-1)^{i})+(\rho_{i}+\rho_{r-i})(n-1)^{i}\\
&\le \binom{r-1}{i-1}(n-1)^{r-i}-\binom{r-1}{i-1}(n-1)^i+\binom{r}{i}(n-1)^i\\
&=\binom{r-1}{i-1}(n-1)^{r-i}+\binom{r-1}{r-i-1}(n-1)^{i}.
\end{aligned}
$$

These inequalities together with (\ref{eq:0&1}) imply
$$
|\F|\le \sum_{2\le j\le r-2}\binom{r-1}{j-1}(n-1)^{r-j}+\binom{r}{r-1}(n-1)+\binom{r}{r}
= n^{r-1}-(n-1)^{r-1}+n-1.
$$
% showing that $|\F|$ is at most the size of the Hilton-Milner-type family.
\end{proof}

Now we are ready to prove Theorem \ref{thm:s=1}.
\begin{proof}[Proof of Theorem \ref{thm:s=1}]
Let $\F\in[n]^{r}$ be non-trivial intersecting. We apply iteratively to $\F$ coordinate-wise shifts that do not destroy non-triviality until no more such shift can be applied. During the process the number of edges in $\F$ and the non-trivial $t$-intersecting property are maintained. In the end, for each $1\le \ell\le r$, $\F$ is either $\ell$-shifted or $\ell$-shift-resistant. 

If $\F$ is not coordinate-wise shifted, then by Lemma~\ref{lem:not_shifted} and Lemma~\ref{lem:I1}, 
$$
|\F|\le 3n^{r-2}-2n^{r-3}\le n^{r-1}-(n-1)^{r-1}+n-1.
$$

If $\F$ is coordinate-wise shifted, then by Lemma~\ref{lem:shifted} we also have 
$$
|\F|\le n^{r-1}-(n-1)^{r-1}+n-1.
$$

\end{proof}

\section{Sunflower Method}\label{sec:sunflower}
A collection of sets $F_1,\dots,F_t$ form a \textit{sunflower ($\Delta$-system)} with $t$ petals if, letting $C=F_1\cap\dots\cap F_t$, the sets $F_i\setminus C$ are pairwise disjoint. This concept was introduced by Erd\H{o}s and Rado~\cite{erdos1960intersection}, who established the following classic upper bound for the size of any set system excluding a sunflower of a given size.

\begin{lem}[Erd\H{o}s-Rado, \cite{erdos1960intersection}]\label{lem: Erdos-Rado}
Let $\mathcal{F}$ be a collection of sets such that $|F| \le r$ for all $F \in \mathcal{F}$. If $\mathcal{F}$ does not contain a sunflower with $t$ petals, then $|\mathcal{F}| \le r! \, (t-1)^r$.
\end{lem}

Building on this result, Frankl~\cite{frankl1978erdos} developed the Sunflower Method to prove stability for the Erd\H{o}s-Ko-Rado theorem. In this section we use the Sunflower Method to prove Theorem \ref{thm:matching_n_large}.

Given a set family $\F \subseteq 2^{[n]}$, a set $A \in \F$, and a nonempty proper subset $C \subsetneq A$, the operation of \emph{shrinking} $\F$ with respect to $C$ produces a new family $\F'$ by replacing all supersets of $C$ in $\F$ with the single set $C$. 
For example, let $\F=\{\{1,2\},\{1,3,4\},\{1,3,5\},\{2,3\},\{3,4,5\}\}$. Shrinking $\F$ with respect to $\{1\}$ removes the supersets $\{1,2\}$, $\{1,3,4\}$, and $\{1,3,5\}$ and adds $\{1\}$, resulting in $\F'=\{\{1\},\{2,3\},\{3,4,5\}\}$.

\begin{defn}
Given $\F \subseteq 2^{[n]}$, a \emph{base} $\B$ of $\F$ is a set family such that
\begin{itemize}
    \item $\B$ is an antichain, i.e. for any distinct $B$ and $B'\in\B$, $B\not\subset B'$;
    \item for any $F\in\F$ there exists $B\in\B$ such that $B\subseteq F$;
    \item $\nu(\B)=\nu(\F)$;
    \item for any $C\subsetneq B\in\B$, shrinking $\B$ with respect to $C$ increases the matching number.
\end{itemize}
\end{defn}

One can construct a base via successive shrinking. Continuing the previous example, we can shrink $\F'$ further with respect to $\{3\}$ to produce the base $\B=\{\{1\},\{3\}\}$, as the matching number remains $2$. However, we cannot shrink $\F'$ with respect to $\{4\}$ because the resulting family $\{\{1\},\{2,3\},\{4\}\}$ has a matching number of $3$, which violates the non-increasing condition.

% The following properties of a base is immediate by definition.
% \begin{prop}\label{prop:base}
% Let integer $r\ge 3$, and let $\F\subseteq\binom{[n]}{r}$. Let $\B$ be a base of $\F$. Then
% \begin{itemize}
%     \item[(i)] for any $F\in\F$ there exists $B\in\B$ such that $B\subseteq F$;
%     \item[(ii)] $\B$ is an antichain, i.e. for any distinct $B$ and $B'\in\B$, $B\not\subset B'$.
% \end{itemize}
% \end{prop}

\begin{lem}\label{lem:base}
Let $r \ge 3$ and $n > s \ge 2$ be integers. Let $\mathcal{F}$ be an $r$-partite $r$-graph with partite sets $X_1, \dots, X_r$ such that $\nu(\mathcal{F}) = s < \tau(\mathcal{F})$. If $\mathcal{B}$ is a base of $\mathcal{F}$, then:
\begin{itemize}
    \item[(i)] $\nu(\mathcal{B}) = s < \tau(\mathcal{B})$;
    \item[(ii)] $\mathcal{B}$ contains no sunflower with $rs+1$ petals, and consequently $|\mathcal{B}| \le r! \, (rs)^r$.
\end{itemize}
\end{lem}

\begin{proof}
For (i), by the definition of a base, we have $\nu(\mathcal{B}) = \nu(\mathcal{F}) = s$. Suppose for the sake of a contradiction that $\tau(\mathcal{B}) = s$. Then there exists a transversal of $\mathcal{B}$ of size $s$. Since every set in $\mathcal{F}$ contains a set in $\mathcal{B}$, any transversal of $\mathcal{B}$ is also a transversal of $\mathcal{F}$. This implies $\tau(\mathcal{F}) \le s$, contradicting the assumption that $\tau(\mathcal{F}) > s$. Thus, we must have $\tau(\mathcal{B}) > s$.

For (ii), suppose to the contrary that $\mathcal{B}$ contains a sunflower with $rs+1$ petals, denoted by $F_1, \dots, F_{rs+1}$. Let $C = \bigcap_{i=1}^{rs+1} F_i$ be the core of the sunflower. 
If $C = \emptyset$, then the sets $F_1, \dots, F_{rs+1}$ are pairwise disjoint, forming a matching of size $rs+1$. Since $r \ge 1$ and $s \ge 1$, this exceeds $s$, contradicting $\nu(\mathcal{B}) = s$. 

Thus, we assume $C \ne \emptyset$. Since $\mathcal{B}$ is a base, it cannot be shrunk further without increasing the matching number. Specifically, shrinking $\mathcal{B}$ with respect to $C$ must increase the matching number (otherwise we would have performed this operation). This implies that there exists a matching $\{M_1, \dots, M_{s+1}\}$ in the family obtained by shrinking $\mathcal{B}$ with respect to $C$. Without loss of generality, let $M_{s+1} = C$. Then $\{M_1, \dots, M_s\}$ must be a matching in $\mathcal{B}$ consisting of sets disjoint from $C$.

Let $U = \bigcup_{j=1}^s M_j$. Since $\{M_1, \dots, M_s\}$ is a matching of size $s$ in $\mathcal{B}$ and $\nu(\mathcal{B})=s$, we cannot extend this matching. Therefore, every other set in the sunflower must intersect $U$. Specifically, for each $i \in \{1, \dots, rs+1\}$, the petal $F_i \setminus C$ must intersect $U$.
However, the sets $\{F_i \setminus C\}_{i=1}^{rs+1}$ are pairwise disjoint. This implies that $U$ contains at least $rs+1$ distinct elements (one from each disjoint petal). 
On the other hand, since $\mathcal{F}$ is an $r$-graph, $|M_j| \le r$, so $|U| \le \sum_{j=1}^s |M_j| \le rs$. 
This gives the contradiction $rs+1 \le |U| \le rs$.

Therefore, $\mathcal{B}$ contains no sunflower with $rs+1$ petals. By Lemma~\ref{lem: Erdos-Rado}, we conclude that $|\mathcal{B}| \le r! \, (rs)^r$.
\end{proof}

To prove Theorem~\ref{thm:matching_n_large}, we establish the following stronger result.

\begin{thm}\label{thm:matching_n_large'}
Let $r\ge 3$ and $s\ge 1$ be integers. There exist $\delta > 0$ and $n>0$ such that for all $n_1 \ge \dots \ge n_r$ satisfying $n_r \ge n$ and $n_1 \le (1+\delta) n_r$, the following holds:

If $r \neq 4$, then
$$
m_0(s,n_1,\dots,n_r) = s\prod_{\ell=1}^{r-1}n_\ell - \prod_{\ell=1}^{r-1}(n_\ell-1) + n_r - s.
$$

If $r = 4$, then
$$
m_0(s,n_1,\dots,n_4) = \max\left\{ s\prod_{\ell=1}^{3}n_\ell - \prod_{\ell=1}^{3}(n_\ell-1) + n_4 - s, \; n_1\left(sn_2n_3 - (n_2-1)(n_3-1) + n_4 - s\right) \right\}.
$$
\end{thm}

Theorem~\ref{thm:matching_n_large} follows immediately from this result by setting $n_1 = \dots = n_r$.

\begin{proof}
We first describe the lower bound constructions. Recall that 
$$
W_r=\{\{\ell,r\}:1\le \ell\le r-1\}\cup\{\{1,\dots,r-1\}\}.
$$
Note that
$$
W_r(n_1,\dots,n_r)\cup\l(\bigcup_{i=2}^s [n_1]\times\dots\times[n_{r-1}]\times\{i\} \r)
$$
is an $r$-partite $r$-graph with $\nu(\F)=s<\tau(\F)$ and its size is $s\prod_{\ell=1}^{r-1}n_\ell - \prod_{\ell=1}^{r-1}(n_\ell-1) + n_r - s$. 

For the 4-uniform case, consider $T=\{\{1,2\}, \{1,3\},\{2,3\}\}$ as a hypergraph on $[4]$. Then
$$
T(n_1,n_2,n_3,n_4)\cup\l(\bigcup_{i=2}^s [n_1]\times[n_2]\times[n_{3}]\times\{i\} \r)
$$
is a $4$-partite $4$-graph with $\nu(\F)=s<\tau(\F)$ and its size is $n_1\left(sn_2n_3 - (n_2-1)(n_3-1) + n_4 - s\right)$.

The constructions described above give the lower bounds.

For upper bounds, let $\F$ be an $r$-partite $r$-graph with partite sets $X_1,\dots,X_r$, where $|X_\ell|=n_\ell$ for $1\le \ell\le r$. Assume that $\nu(\F)=s<\tau(\F)$, and $|\F|=m_0(s,n_1,\dots,n_r)$; that is, $\F$ has maximum size among all $r$-partite $r$-graphs with these properties. We need to show that if $r\not=4$, then
\begin{equation}\label{eq:r not 4}
|\F|\le s\prod_{\ell=1}^{r-1}n_\ell - \prod_{\ell=1}^{r-1}(n_\ell-1) + n_r - s,
\end{equation}
and if $r=4$, then
\begin{equation}\label{eq:r = 4}
|\F|\le \max\left\{ s\prod_{\ell=1}^{3}n_\ell - \prod_{\ell=1}^{3}(n_\ell-1) + n_4 - s, \; n_1\left(sn_2n_3 - (n_2-1)(n_3-1) + n_4 - s\right) \right\}.
\end{equation}

Let $\B$ be a base of $\F$ and let $\rho_i$, $1\le i\le r$, be the number of $i$-sets in $\B$. For a set $Y$, $|Y|\le r$ define $\F(Y)=\{F\in\F:Y\subset F\}$ and note $\F(Y)=\emptyset$ for $Y$ with $\max|Y\cap X_i|\ge 2$. Define 
$$
n(Y)=\prod\{n_\ell:Y\cap X_\ell=\emptyset\}.
$$
Let $\eps>0$ be a sufficiently small constant. By Lemma~\ref{lem:base} $(i)$ and the maximality of $\F$, $|\F(B)|=n(B)$ for any $B\in\B$. Let $\delta\ll\eps$. By definition of a base, and the inequalities $n_i\le n_1\le (1+\delta)n_r$, 
\begin{equation}\label{eq:upper}
|\F|\le \sum_{B\in\B}|\F(B)|= \sum_{B\in\B}n(B)\le \sum_{i=1}^r\rho_i((1+\delta)n_r)^{r-i}\le (1+\eps)\sum_{i=1}^r\rho_in_r^{r-i}.
\end{equation}

On the other hand, let $n\gg \frac{1}{\eps}$. Note that for any $Y_1\subsetneq Y_2$, $\frac{n(Y_2)}{n(Y_1)}\le \frac{1}{n_r}$. Using this, together with Lemma~\ref{lem:base} $(ii),$ and $n_r\ge n$,
\begin{equation}\label{eq:lower}
\begin{aligned}
|\F|&\ge \sum_{B\in\B}|\F(B)|-\sum_{B,B'\in\B}|\F(B\cup B')|\\
&=\sum_{B\in\B}\l(|\F(B)|-\frac{1}{2}\sum_{B'\in\B, B'\not=B}|\F(B\cup B')|\r)\\
&=\sum_{B\in\B}\l(1-\frac{1}{2}\sum_{B'\in\B, B'\not=B}\frac{n(B\cup B')}{n(B)}\r)n(B)\\
&\ge \sum_{B\in\B}\l(1-\frac{|\B|}{2n_r}\r)n(B)\\
&\ge (1-\eps)\sum_{i=1}^r\rho_in_r^{r-i}.
\end{aligned}    
\end{equation}

To maximize $|\F|$, we need $\rho_1$ to be as large as possible. Since $\nu(\B)=s$, we have $\rho_1\le s$. Suppose for contradiction that $\rho_1=s$. Then $\B$ contains $s$ singletons and no set of size larger than 1. Indeed, if $B\in\B$ is a set of size at least 2, since $\B$ is an antichain, $B$ does not contain any of the $s$ singletons, and hence $B$ together with the $s$ singletons would form a matching of size $s+1$ of $\B$, contradicting $\nu(\B)=s$. Hence $\B$ consists of $s$ singletons only, implying that $\tau(\B)=s$, which contradicts Lemma~\ref{lem:base} $(i)$. Therefore, $\rho_1<s$, and by maximality of $|\F|$, $\rho_1=s-1$.

We will prove by induction on $s$. When $s=1$, $\rho_1=0$. To maximize $|F|$, by (\ref{eq:upper}) and (\ref{eq:lower}), we want $\rho_2$ to be as large as possible. Since the 2-sets in $\B$ is intersecting, they either form a triangle or a star. We split the proof into two cases.

\noindent\textbf{Case 1: If the 2-sets in $\B$ form a triangle}, then since $\B$ is an antichain and $\nu(\B)=1$, $\B$ contains no set of size larger than 2. Thus $\B$ contains three 2-sets only. It is clear that $|\F|$ achieves maximum when the three vertices of the triangle lie in $X_{r-2},X_{r-1}$ and $X_{r}$, and in such case, 
$$
|\F|=\prod_{\ell=1}^{r-3}n_\ell(n_{r-2}+n_{r-1}+n_{r}-2).
$$
When $r=3$, $|\F|=n_1+n_2+n_3-2$, satisfying (\ref{eq:r not 4}); when $r=4$, $|\F|=n_1(n_{2}+n_{3}+n_{4}-2)$, satisfying (\ref{eq:r = 4}).\\
When $r\ge 5$, since $n_1\le (1+\delta)n_r$ where $\delta$ is sufficiently small, and $n_r\ge n$ being sufficiently large, by viewing all $n_i$ as $(1+o(1))n_r$ we have
$$
|\F|=\prod_{\ell=1}^{r-3}n_\ell(n_{r-2}+n_{r-1}+n_{r}-2)<\prod_{\ell=1}^{r-1}n_\ell-\prod_{\ell=1}^{r-1}(n_\ell-1)+n_r-1,
$$
satisfying (\ref{eq:r not 4}).

\noindent\textbf{Case 2: If the 2-sets in $\B$ form a star}, let those two sets be $\{x,y_i\}$, $1\le i\le t$. Since $\tau(\B)>1$, there exists $B\in\B$ such that $x\not\in B$, and since $\nu(\B)=1$, $y_i\in B$ for all $1\le i\le t$. Thus the $y_i$'s lie in different parts, and hence the number of 2-sets is at most $r-1$. By maximality of $|\F|$, we have $\rho_2=r-1$. In such case, $|\F|$ achieves maximum when 
$$
\B=\{\{x,y_\ell\}:1\le \ell\le r-1\}\cup\{\{y_1,\dots,y_{r-1}\}\},
$$
where $x\in X_r$ and $y_{\ell}\in X_\ell$ for $1\le \ell\le r-1$ (when $r=3$, this reduces to case 1). Hence,
$$
|\F|=\prod_{\ell=1}^{r-1}n_\ell-\prod_{\ell=1}^{r-1}(n_\ell-1)+n_r-1,
$$
satisfying (\ref{eq:r not 4}) and (\ref{eq:r = 4}). This completes the proof of the $s=1$ case.

When $s\ge 2$, since $\rho_1=s-1\ge 1$, $\B$ has at least one singleton $y$. Let $\F'$ be the $r$-partite $r$-graph obtained by deleting all sets containing $y$ from $\F$, and let $\B'$ be the collection of sets obtained by deleting $\{y\}$ from $\B$. Then $\nu(\F')=\nu(\B'\setminus\{\{y\}\})=s-1$. Indeed, if $\F'$ has a matching of size $s$, then $\B'$ has a matching of size $s$, which together with $\{y\}$ form a matching of size $s+1$ of $\B$, contradicting $\nu(\B)=s$. Note that $\tau(\F')\ge \tau(\F)-1> s-1>\nu(\F')$. Thus we may upper bound the size of $|\F|$ by applying inductive hypothesis on $\F'$ and adding back the edges containing $y$.

\noindent\textbf{Case 1: $r\not=4$.}

\noindent\textbf{Case 1.1: If $y$ is contained in a part of size $|X_r|$}, without loss of generality assume $y\in X_r$, then
$$
|\F|\le (s-1)\prod_{\ell=1}^{r-1}n_\ell-\prod_{\ell=1}^{r-1}(n_\ell-1)+(n_r-1)-(s-1)+\prod_{\ell=1}^{r-1}n_\ell=s\prod_{\ell=1}^{r-1}n_\ell-\prod_{\ell=1}^{r-1}(n_\ell-1)+n_r-s,
$$
satisfying (\ref{eq:r not 4}).

\noindent\textbf{Case 1.2: If $y$ is contained in a part of size strictly larger than $|X_r|$}, without loss of generality assume $y\in X_1$, then
 $$
 \begin{aligned}
|\F|&\le (s-1)(n_1-1)\prod_{\ell=2}^{r-1}n_\ell-(n_1-2)\prod_{\ell=2}^{r-1}(n_\ell-1)+n_r-(s-1)+\prod_{\ell=2}^{r}n_\ell\\
&=\l(s\prod_{\ell=1}^{r-1}n_\ell-\prod_{\ell=1}^{r-1}(n_\ell-1)+n_r-s\r)-\prod_{\ell=1}^{r-1}n_\ell-(s-1)\prod_{\ell=2}^{r-1}n_\ell+\prod_{\ell=2}^{r-1}(n_\ell-1)+1+\prod_{\ell=2}^{r}n_\ell\\
&\le s\prod_{\ell=1}^{r-1}n_\ell-\prod_{\ell=1}^{r-1}(n_\ell-1)+n_r-s,
 \end{aligned}
$$
where the last inequality uses $n_1\ge\dots\ge n_r\ge 1$ and $s\ge 2$. This satisfies (\ref{eq:r not 4}).

\noindent\textbf{Case 2: $r=4$.}

Recall that we want
\begin{equation}\label{eq:case2}
|\F|\le \max\{s\prod_{\ell=1}^{3}n_\ell-\prod_{\ell=1}^{3}(n_\ell-1)+n_4-s, n_1(sn_2n_3-(n_2-1)(n_3-1)+n_4-s)\}.   
\end{equation}

\noindent\textbf{Case 2.1: If $y$ is contained in a part of size $|X_4|$}, without loss of generality assume $y\in X_4$, then 
\begin{equation}\label{eq:Case2.1}
\begin{aligned}
|\F|\le \max\{&(s-1)\prod_{\ell=1}^{3}n_\ell-\prod_{\ell=1}^{3}(n_\ell-1)+(n_4-1)-(s-1),\\
&n_1((s-1)n_2n_3-(n_2-1)(n_3-1)+(n_4-1)-(s-1))\}+n_1n_2n_3.  
\end{aligned}
\end{equation}
Subtracting the first term in the max in (\ref{eq:Case2.1}) by the first term in the max in (\ref{eq:case2}), 
$$
\begin{aligned}
&(s-1)\prod_{\ell=1}^{3}n_\ell-\prod_{\ell=1}^{3}(n_\ell-1)+(n_4-1)-(s-1)-\l(s\prod_{\ell=1}^{3}n_\ell-\prod_{\ell=1}^{3}(n_\ell-1)+n_4-s\r)\\
=&-n_1n_2n_3.
\end{aligned}
$$
Subtracting the second term in the max in (\ref{eq:Case2.1}) by the second term in the max in (\ref{eq:case2}), 
$$
\begin{aligned}
&n_1(sn_2n_3-(n_2-1)(n_3-1)+n_4-s)-n_1((s-1)n_2n_3-(n_2-1)(n_3-1)+(n_4-1)-(s-1))\\
=&-n_1n_2n_3
\end{aligned}
$$
Thus (\ref{eq:Case2.1}) is equivalent to (\ref{eq:case2}).

\noindent\textbf{Case 2.2: If $y$ is contained in a part of size strictly larger than $|X_4|$}, there are essentially two cases: $y\in X_1$ and $|X_1|>|X_2|$, or $y\in X_2$ (the case $y\in X_3$ is the same as $y\in X_2$).

\noindent\textbf{Case 2.2.1: When $y\in X_1$ and $|X_1|>|X_2|$}, 
\begin{equation}\label{eq:case2.2.1}
\begin{aligned}
|\F|\le \max\{&(s-1)(n_1-1)n_2n_3-(n_1-2)(n_2-1)(n_3-1)+n_4-(s-1),\\
&(n_1-1)((s-1)n_2n_3-(n_2-1)(n_3-1)+n_4-(s-1))\}+n_2n_3n_4.
\end{aligned}    
\end{equation}

Subtracting the first term in the max in (\ref{eq:case2.2.1}) by the first term in the max in (\ref{eq:case2}),
$$
\begin{aligned}
&(s-1)(n_1-1)n_2n_3-(n_1-2)(n_2-1)(n_3-1)+n_4-(s-1)-\l(s\prod_{\ell=1}^{3}n_\ell-\prod_{\ell=1}^{3}(n_\ell-1)+n_4-s\r)\\
=&-n_1n_2n_3-(s-1)n_2n_3+(n_2-1)(n_3-1)+1\\
\le& -n_2n_3n_4,
\end{aligned}
$$
where the last inequality uses $n_1\ge n_2\ge n_3\ge n_4\ge 1$ and $s\ge 2$.

Subtracting the second term in the max in (\ref{eq:case2.2.1}) by the second term in the max in (\ref{eq:case2}),
$$
\begin{aligned}
&(n_1-1)((s-1)n_2n_3-(n_2-1)(n_3-1)+n_4-(s-1))-n_1(sn_2n_3-(n_2-1)(n_3-1)+n_4-s)\\
=&-((s-1)n_2n_3-(n_2-1)(n_3-1)+n_4-(s-1)) -n_1(n_2n_3-1)\\
=&-n_1n_2n_3-(s-2)n_2n_3+n_1-n_2-n_3-n_4+s-1\\
\le& -n_2n_3n_4,
\end{aligned}
$$
where the last inequality uses $(1+\delta)n_4\ge n_1\ge n_2\ge n_3\ge n_4\ge 1$, $\delta$ being sufficiently small, and $s\ge 2$.

Therefore, (\ref{eq:case2.2.1}) implies (\ref{eq:case2}).

\noindent\textbf{Case 2.2.2: When $y\in X_2$},
\begin{equation}\label{eq:case2.2.2}
\begin{aligned}
|\F|\le \max\{&(s-1)n_1(n_2-1)n_3-(n_1-1)(n_2-2)(n_3-1)+n_4-(s-1),\\
&n_1((s-1)(n_2-1)n_3-(n_2-2)(n_3-1)+n_4-(s-1))\}+n_1n_3n_4.
\end{aligned}    
\end{equation}

Subtracting the first term in the max in (\ref{eq:case2.2.2}) by the first term in the max in (\ref{eq:case2}),
$$
\begin{aligned}
&(s-1)n_1(n_2-1)n_3-(n_1-1)(n_2-2)(n_3-1)+n_4-(s-1)-\l(s\prod_{\ell=1}^{3}n_\ell-\prod_{\ell=1}^{3}(n_\ell-1)+n_4-s\r)\\
=&-n_1n_2n_3-(s-1)n_1n_3+(n_1-1)(n_3-1)+1\\
\le& -n_1n_3n_4,
\end{aligned}
$$
where the last inequality uses $n_1\ge n_2\ge n_3\ge n_4\ge 1$ and $s\ge 2$.

Subtracting the second term in the max in (\ref{eq:case2.2.2}) by the second term in the max in (\ref{eq:case2}),
$$
\begin{aligned}
&n_1((s-1)(n_2-1)n_3-(n_2-2)(n_3-1)+n_4-(s-1))-n_1(sn_2n_3-(n_2-1)(n_3-1)+n_4-s)\\
=&-n_1(n_2n_3+(s-1)n_3-(n_3-1)-1)\\
=&-n_1n_2n_3-(s-2)n_1n_3\\
\le& -n_2n_3n_4,
\end{aligned}
$$
where the last inequality uses $n_1\ge n_4\ge 1$ and $s\ge 2$.

Therefore, (\ref{eq:case2.2.2}) implies (\ref{eq:case2}). This completes the proof.

\end{proof}

\section{Maximum $(r-2)$-intersecting $r$-partite family}\label{sec:r-2}
This section proves Theorem \ref{thm:(r-2)-intersecting}. We first prove the case $r=3$, i.e. Theorem \ref{thm:3-partite intersecting}.

\begin{proof}[Proof of Theorem \ref{thm:3-partite intersecting}]
Let $\F\subset [n_1]\times [n_2]\times [n_3]$ be intersecting and non-trivial. 

\noindent\textbf{Case 1:} Suppose $\F$ contains a copy of $K_3(2)$, i.e. there are three edges $(v_1,v_2,w_3)$, $(v_1,w_2,v_3)$, $(w_1,v_2,v_3)$ such that $v_i\not=w_i$ for $1\le i\le 3$. If there exists an $F\in \F$ such that $F\cap(v_1,v_2,v_3)=\emptyset$, then since $\nu(\F)=1$, we have $F=(w_1,w_2,w_3)$. It is not hard to check that there cannot be any other edges in $\F$, and hence $|\F|=4\le n_1+n_2+n_3-2$. Thus we can assume that every edge $F$ of $\F$ has a non-empty intersection with $(v_1,v_2,v_3)$. For any $F\in \F$, Without loss of generality, assume $F[1]=v_1\not=w_1$. Note that $F\cap (w_1,v_2,v_3)\not=\emptyset$, so either $F[2]=v_2$ or $F[3]=v_3$. This means that every edge $F$ agrees with $(v_1,v_2,v_3)$ in at least two coordinates. The upper bound follows by the fact that the number of such edges in $[n_1]\times [n_2]\times [n_3]$ is at most $(n_1-1)+(n_2-1)+(n_3-1)+1=n_1+n_2+n_3-2$.

\noindent\textbf{Case 2:} Suppose $\F$ is $K_3(2)$-free. Let $F_1=(u_1,u_2,u_3)\in\F$. Assume for contradiction that $e(H)\ge n_1+n_2+n_3-1$. Then there must exist an edge $F_2$ of $H$ intersecting $F_1$ in exactly one vertex (because the number of edges intersecting $F_1$ in at least two vertices is at most $n_1+n_2+n_3-2$). Without loss of generality, let $F_2=(u_1,x_2,x_3)$. For any other edge $F\in\F$, if $F[1]\not=u_1$, then there exists $x_1\not=u_1$ such that $F[1]=x_1$. Since $F$ intersects $F_1$ and $F_2$ in at least one vertex, we have either $e=(x_1,x_2,u_3)$ or $e=(x_1,u_2,x_3)$, either case $F_1,F_2$ and $F$ form a copy of $K_3(2)$, a contradiction. Thus every edge of $\F$ has $u_1$ as first coordinate. This contradicts $\F$ being non-trivial, and hence completes the proof.
\end{proof}

Next, we prove the case $r\ge 3$.

\begin{proof}[Proof of Theorem \ref{thm:(r-2)-intersecting}]
Let $\F\in [n_1]\times\dots\times[n_r]$ be $(r-2)$-intersecting and non-trivial.
We prove by induction on $r$. The case $r=3$ is Theorem \ref{thm:3-partite intersecting}. When $r\ge4$, suppose the statement holds for $r-1$. If there exists a vertex $v$ contained in every edge, then by deleting $v$ from all edges and deleting the part containing $v$, we obtain an $(r-3)$-intersecting nontrivial $(r-1)$-partite $(r-1)$-graph $\F'$. By inductive hypothesis we have 
$$
|\F|=|\F'|\le \sum_{i=1}^r{n_i}-\min_{1\le i\le r}\{n_i\}-(r-1)+1\le  \sum_{i=1}^r{n_i}-r+1.
$$
Thus we can assume no vertex is contained in all edges. There exist two edges $F_1$ and $F_2$ such that $|F_1\cap F_2|=r-2$, because otherwise every edge of $\F$ would contain a fixed set of $r-1$ vertices, contradicting the fact that $\F$ is non-trivial. Without loss of generality, let $F_1=(a_1,a_2, v_3,v_4,\dots, v_r)$ and $F_2=(b_1,b_2,v_3,v_4\dots, v_r)$. For $3\le i\le r$, let $F_i$ be an edge such that $F_i[i]=x_i\not=v_i$ (such an edge exists because no vertex is contained in every edge). Since $|F_3\cap F_1|\ge r-2$ and $|F_3\cap F_2|\ge r-2$, we have either $F_3=(a_1,b_2,x_3,v_4,\dots,v_r)$ or $F_3=(b_1,a_2,x_3,v_4,\dots,v_r)$. Without loss of generality, let $F_3=(a_1,b_2,x_3,v_4,\dots,v_r)$. Then for $i\ge 4$, since $|F_i\cap F_j|\ge r-2$ for every $1\le j\le 3$, we have $F_i=(a_1,b_2,v_3,\dots,v_{i-1},x_i,v_{i+1},\dots,v_r)$. One can check that the edges $F_1,F_2,\dots,F_r$ form a copy of $K^{r-1}_r[2]$, since each of them agree with $(a_1,b_2,v_3,\dots,v_r)$ in exactly $r-1$ coordinates.

For convenience, we relabel the vertices in the copy of $K^{r-1}_r[2]$: for $1\le i\le r$, let $F_i[i]=x_i$ and $F_i[j]=v_j$ for every $j\not=i$ where $x_k\not =v_k$ for every $1\le k\le r$. For any other edge $F$ in $\F$, we will show that $|F\cap (v_1,\dots, v_r)|\ge r-1$. Indeed, if $|F\cap (v_1,\dots, v_r)|\le r-2$, without loss of generality we assume $F[1]\not=v_1$ and $F[2]\not=v_2$. Note that $F_3=(v_1,v_2,\dots)$, and $|F\cap F_3|\ge r-2$, we have $F[i]=F_3[i]$ for every $3\le i\le r$. In particular, $F[3]=F_3[3]=x_3$. By the same argument, we also have $F[3]=F_4[3]=v_3$, which contradicts $x_3\not=v_3$. Thus $|F\cap (v_1,\dots, v_r)|\ge r-1$ holds for every edge $F$ in $H$. It is not hard to check that the number of such edges is at most $\sum_{i=1}^r{n_i}-r+1$, which completes the proof.
\end{proof}

\section{Intersection problem for $n$ sufficiently large}\label{sec:intersection}
This section proves Theorem \ref{thm:t-intersecting}.

Consider two constructions of $t$-intersecting families in $\binom{[r]}{t+1}$. Let 
$$
S^{t+1}_r=\{[t]\cup\{i\}:t+1\le i\le r\},
$$ 
and let 
$$
K^{t+1}_{t+2}=\binom{[t+2]}{t+1}.
$$
The following lemma shows that the maximum $t$-intersecting families of $(t+1)$-sets is achieved by either $S^{t+1}_r$ or $K^{t+1}_{t+2}$.

\begin{lem}\label{lem:t-intersecting (t+1)-graph}
Let $r-2\ge t \ge 1$. Let $\cH\in \binom{[r]}{t+1}$ be a $t$-intersecting family. Then 
$$|\cH|\le \max\{r-t,t+2\},$$
and equality holds only when $\cH\cong S^{t+1}_r$ or $\cH\cong K^{t+1}_{t+2}$.
\end{lem}

\begin{proof}
We will prove by induction on $t\ge1$. When $t=1$, it is clear that a maximum intersecting family of 2-sets on $[r]$ has size $\max\{r-1,3\}$, achieved only by a star or a triangle. When $t\ge 2$, let $\cH$ be a $t$-intersecting family of $(t+1)$-sets on $[r]$.

\noindent\textbf{Case 1:} there exists a vertex contained in every edge of $\cH$. Without loss of generality we assume that this common vertex is $r$. Deleting $r$ from each edge of $\cH$ produces a $(t-1)$-intersecting family $\cH'$ of $t$-sets on $[r-1]$. By the inductive hypothesis, 
$$
|\cH|=|\cH'|\le \max\{(r-1)-(t-1),(t-1)+2\}\le \max\{r-t, t+2\}.
$$
Equality holds only when $|\cH|=|\cH'|=r-t$, and if that happens, then $\cH'\cong S^t_{r-1}$ and hence $\cH\cong S^{t+1}_r$.

\noindent\textbf{Case 2:} no vertex lies in every edge of $\cH$. Without loss of generality, let $E_1=[t+2]\setminus\{1\}$ and $E_2=[t+2]\setminus\{2\}$ be two edges in $\cH$. For each $3\le i\le t+2$, there exists an edge $E_i$ not containing $i$ (otherwise $i$ would belong to every edge, contradicting the assumption). Note that $E_i$ must intersect $E_1$ in $t$ elements, all but at most one vertex in $E_1$ is not contained in $E_i$, it follows that $E_1\setminus\{i\}\subseteq E_i$. Similarly $E_2\setminus\{i\}\subseteq E_i$. These two inclusions force $E_i=[t+2]\setminus\{i\}$. Thus every $3\le i\le t+2$ yields the edge $E_i=[t+2]\setminus\{i\}$. Note that no other $(t+1)$-subset of $[r]$ can belong to $\cH$ because any additional distinct $(t+1)$-set would fail to be $t$-intersecting with some existing one. Therefore $|\cH|\le t+2$, with equality only when $\cH\cong K^{t+1}_{t+2}$. 
\end{proof}

We are now ready to prove Theorem~\ref{thm:t-intersecting}.
\begin{proof}[Proof of Theorem \ref{thm:t-intersecting}]
By Lemma~\ref{lem:not_shifted}, we may assume that $\F$ is coordinate-wise shifted and has maximum size. Then $\P(\F)$ is non-trivial $t$-intersecting by Lemma~\ref{lem:projection_maintain_intersection}. Let $\cH_i$ be the collection of $i$-sets in $\P(\F)$ and let $\rho_i=|\cH_i|$. Since $\F$ has maximum size, we have
\[
    |\F|=\sum_{i=1}^{r}\rho_i(n-1)^{r-i}.
\]
Since $\P(\F)$ is non-trivial $t$-intersecting, $\rho_i=0$ for all $1\le i\le t$. Thus, as $n\rightarrow\infty$, the leading term in $|\F|$ is $\rho_{t+1}n^{r-t-1}$. Note that $\cH_{t+1}$ is $t$-intersecting. By Lemma~\ref{lem:t-intersecting (t+1)-graph}, $\rho_{t+1}\le \max\{r-t,t+2\}$, and this maximum is achieved only by $S^{t+1}_r$ or $K^{t+1}_{t+2}$. We consider these two cases.

\noindent\textbf{Case 1:} $\cH_{t+1}\cong S^{t+1}_r$ and hence $\rho_{t+1}=r-t$. Recall that $S^{t+1}_r=\{[t]\cup\{\ell\}:t+1\le \ell\le r\}$. For any $t+2\le i\le r-2$, if $E\in \cH_i$, then $E$ must intersect $[t]\cup\{\ell\}$ in at least $t$ elements for every $t+1\le \ell\le r$. This condition implies $[t]\subset E$. Thus,
\[
    \rho_i\le \binom{r-t}{i-t}
\]
for each $t+2\le i\le r-2$. Together with $\rho_{r-1}\le r$ and $\rho_r\le 1$, we have
\[
    |\F|\le \sum_{i=t+1}^{r-2}\binom{r-t}{i-t}(n-1)^{r-i}+r(n-1)+1 = n^{r-t}-(n-1)^{r-t}+t(n-1),
\]
where equality holds only when $\F=W_{r,t}(n)$.

\noindent\textbf{Case 2:} $\cH_{t+1}\cong K^{t+1}_{t+2}$ and hence $\rho_{t+1}=t+2$. Recall that $K^{t+1}_{t+2}=\binom{[t+2]}{t+1}$. For any $t+2\le i\le r$ and $E\in \cH_i$, $E$ intersects every $F\in \binom{[t+2]}{t+1}$ in at least $t$ elements, which forces $|E\cap[t+2]|\ge t+1$. Thus,
\[
    \rho_i\le (t+2)\binom{r-t-2}{i-t-1}+\binom{r-t-2}{i-t-2}
\]
for every $t+2\le i\le r$. Hence,
\[
    |\F|\le \sum_{i=t+1}^{r}\left((t+2)\binom{r-t-2}{i-t-1}+\binom{r-t-2}{i-t-2}\right)n^{r-i}=(t+2)n^{r-t-1}-(t+1)n^{r-t-2},
\]
and equality holds only when $\F=K_{r,t}(n)$.
\end{proof}

\bibliographystyle{amsplain}
\bibliography{refs.bib}

@article{hou2026non,
  title={Non-trivial Intersection Problems for Multi-part Hypergraphs},
  author={Hou, Jianfeng and Hu, Caiyun},
  journal={arXiv preprint arXiv:2606.06208},
  year={2026}
}

@misc{bolzano2026casestudies,
  title = {{Bolzano}: Case Studies in {LLM}-Assisted Mathematical Research},
  author = {Balko, Martin and Greb{\'{\i}}k, Jan and Hub{\'a}{\v c}ek, Pavel and Kouteck{\'y}, Martin and Kripner, Mat{\v e}j and Rozho{\v n}, V{\'a}clav and {\v S}{\'a}mal, Robert and Z{\'a}me{\v c}n{\'{\i}}k, Adri{\'a}n},
  year = {2026},
  eprint = {2604.16989},
  archivePrefix = {arXiv},
  primaryClass = {cs.CL},
  doi = {10.48550/arXiv.2604.16989},
  url = {https://arxiv.org/abs/2604.16989}
}

@article{erdos1960intersection,
  title={Intersection theorems for systems of sets},
  author={Erd{\"o}s, Paul and Rado, Richard},
  journal={Journal of the London Mathematical Society},
  volume={1},
  number={1},
  pages={85--90},
  year={1960},
  publisher={Wiley Online Library}
}

@incollection {frankl1987shifting,
    AUTHOR = {Peter Frankl},
     TITLE = {The shifting technique in extremal set theory},
 BOOKTITLE = {Surveys in combinatorics 1987 ({N}ew {C}ross, 1987)},
    SERIES = {London Math. Soc. Lecture Note Ser.},
    VOLUME = {123},
     PAGES = {81--110},
 PUBLISHER = {Cambridge Univ. Press, Cambridge},
      YEAR = {1987},
      ISBN = {0-521-34805-6},
   MRCLASS = {05A05},
  MRNUMBER = {905277},
MRREVIEWER = {E.\ C.\ Milner},
}

@article{frankl1999erdos,
  title={The Erdos-Ko-Rado theorem for integer sequences},
  author={Frankl, Peter and Tokushige, Norihide},
  journal={Combinatorica},
  volume={19},
  number={1},
  pages={55--64},
  year={1999},
  publisher={Budapest: Akademiai Kiado,[1981-}
}

@article{ahlswede1997complete,
  title={The complete intersection theorem for systems of finite sets},
  author={Ahlswede, Rudolf and Khachatrian, Levon H},
  journal={European journal of combinatorics},
  volume={18},
  number={2},
  pages={125--136},
  year={1997},
  publisher={Academic Press}
}

@inproceedings{frankl1978erdos,
  title={The Erdos-Ko-Rado theorem is true for n= ckt},
  author={Frankl, Peter},
  booktitle={Combinatorics (Proc. Fifth Hungarian Colloq., Keszthely, 1976)},
  volume={1},
  pages={365--375},
  year={1978}
}

@article{frankl1980erdos,
  title={The Erd{\"o}s-Ko-Rado theorem for integer sequences},
  author={Frankl, Peter and F{\"u}redi, Zolt{\'a}n},
  journal={SIAM Journal on Algebraic Discrete Methods},
  volume={1},
  number={4},
  pages={376--381},
  year={1980},
  publisher={SIAM}
}

@article{Erdos1965Matching,
  author    = {Paul Erdős},
  title     = {A problem on independent $r$-tuples},
  journal   = {Ann. Univ. Sci. Budapest. Eötvös Sect. Math.},
  volume    = {8},
  year      = {1965},
  pages     = {93--95}
}

@article{HiltonMilner1967,
  author    = {A. J. W. Hilton and E. C. Milner},
  title     = {Some intersection theorems for systems of finite sets},
  journal   = {Quarterly Journal of Mathematics, Oxford Series (2)},
  volume    = {18},
  year      = {1967},
  pages     = {369--384}
}

@article{ErdosKoRado1961,
  author    = {Paul Erdős and Chao Ko and Richard Rado},
  title     = {Intersection theorems for systems of finite sets},
  journal   = {Quarterly Journal of Mathematics, Oxford Series (2)},
  volume    = {12},
  year      = {1961},
  pages     = {313--320}
}

@inproceedings{meyer1974quelques,
  title={Quelques problemes concernant les cliques des hypergraphes h-complets et q-parti h-complets},
  author={Meyer, J},
  booktitle={Hypergraph Seminar},
  pages={127--139},
  year={1974},
  organization={Springer}
}

@article{Konig1931,
  author    = {Dénes Kőnig},
  title     = {Gráfok és mátrixok},
  journal   = {Matematikai és Fizikai Lapok},
  volume    = {38},
  year      = {1931},
  pages     = {116--119},
  language  = {Hungarian}
}

@article {frankl2012disjoint,
    AUTHOR = {Frankl, Peter},
     TITLE = {Disjoint edges in separated hypergraphs},
   JOURNAL = {Mosc. J. Comb. Number Theory},
  FJOURNAL = {Moscow Journal of Combinatorics and Number Theory},
    VOLUME = {2},
      YEAR = {2012},
    NUMBER = {4},
     PAGES = {19--26},
      ISSN = {2220-5438,2640-7361},
   MRCLASS = {05D05 (05C35 05C65 05C70 05D15)},
  MRNUMBER = {3065278},
MRREVIEWER = {J\'ozsef\ Balogh},
}

@article{deza1983erdos,
  title={Erd{\"o}s--ko--rado theorem—22 years later},
  author={Deza, M and Frankl, P},
  journal={SIAM Journal on Algebraic Discrete Methods},
  volume={4},
  number={4},
  pages={419--431},
  year={1983},
  publisher={SIAM}
}

@article{lu2024matching,
  title={Matching stability for 3-partite 3-uniform hypergraphs},
  author={Lu, Hongliang and Ma, Xinxin},
  journal={arXiv preprint arXiv:2410.15673},
  year={2024}
}

\end{document}